\documentclass[11pt]{amsart}
\usepackage{amsmath,amssymb, amscd}
\usepackage[all,dvips]{xypic}
\usepackage{times}
\usepackage{enumerate} 

\newcommand{\colim}{{\rm colim}}
\newcommand{\CB}{{\rm CB}}
\newcommand{\CH}{{\rm CH}}
\newcommand{\Hom}{{\rm Hom}}
\newcommand{\Ext}{{\rm Ext}}

\newcommand{\Ch}{{\rm\bf Ch}}

\newcommand{\tl}{\ensuremath{\triangleleft}}
\newcommand{\tr}{\ensuremath{\triangleright}}

\newcommand{\C}[1]{\ensuremath{\mathcal{#1}}}
\newcommand{\B}[1]{\ensuremath{\mathbb{#1}}}
\newcommand{\lmod}[1]{\ensuremath{{#1}\text{-{\bf Mod}}}}
\newcommand{\rmod}[1]{\ensuremath{\text{{\bf Mod}-}{#1}}}

\newcommand{\xra}[1]{\ensuremath{\xrightarrow{\ {#1}\ }}}
\newcommand{\xla}[1]{\ensuremath{\xleftarrow{\ {#1}\ }}}

\newtheorem{theorem}{Theorem}[section]

\newtheorem{lemma}[theorem]{Lemma}
\newtheorem{proposition}[theorem]{Proposition}

\theoremstyle{definition}
\newtheorem{definition}[theorem]{Definition}
\newtheorem{remark}[theorem]{Remark}

\title{Uniqueness of pairings in Hopf-cyclic cohomology}
\email{kaygun@mpim-bonn.mpg.de}
\address{Max-Planck-Institut f\"ur Mathematik, Vivatsgasse 7, Bonn 53111, Germany}
\author{Atabey Kaygun}

\begin{document}
\maketitle

\begin{abstract}
  We show that all pairings defined in the literature extending
  Connes-Moscovici characteristic map in Hopf-cyclic cohomology are
  isomorphic as natural transformations of derived double functors.
\end{abstract}

\section{Introduction}

The category of algebras over a fixed ground ring is not an abelian
category.  This means we are denied the use of the amenities provided
by the classical homological algebra on the category of algebras and
their morphisms directly.  One way around this problem is to find a
``good'' (faithful) functor from the category of algebras into an
abelian category and apply the tools of homological algebra on the
image of this functor.  Hochschild and cyclic (co)homology are
examples of this type where we use $\lmod{\Delta}$ the abelian
category of simplicial modules for the former and $\lmod{\Lambda}$ the
abelian category of cyclic modules for the latter.

One recurring problem in homological algebra is the task of showing
certain functors $F_*, G_*\colon {\bf D}(\C{A})\to \lmod{k}$ on the
derived category ${\bf D}(\C{A})$ of an abelian category $\C{A}$ are
isomorphic.  This task can be accomplished by finding isomorphic
functors $F$ and $G$ which are defined on the underlying abelian
category $\C{A}$, and showing that $F_*$ and $G_*$ are obtained as
{\em derivatives} of these functors.  If our functors $F$ and $G$ are
defined only on a non-abelian subcategory, this approach will not
work.  The example we have in mind is the category of (co)cyclic
modules and the subcategory of (co)cyclic modules coming from the
image of the functor $Z\mapsto Cyc_\bullet(Z)$ sending an (co)algebra
to its canonical (co)cyclic module.

In the context of cyclic (co)homology of (co)algebras there are
various derived categories used in the literature.  We note (i) the
derived category of cyclic modules~\cite{Connes:ExtFunctors}, (ii) the
derived category of mixed
complexes~\cite{JonesKassel:BivariantCyclicTheory} and (iii) the
homotopy category of towers of super
complexes~\cite{CuntzQuillen:NonsingularityII} where the last two of
are homotopy equivalent by~\cite{Quillen:CyclicHomologyType}.  In this
paper, we add another derived category to this list: ${\bf
  D}(\lmod{(\Lambda,{\bf T})})$ the relative derived category of
cyclic modules to implement Connes' very first definition of cyclic
cohomology \cite[Sect. I, Def. 2]{Connes:NonCommutativeGeometry} as a
derived functor in Theorem~\ref{Equivalent}.  This derived category
will allow us to form a bridge between isomorphic pairings defined in
(i) and (ii)~\cite{Kaygun:CupProduct}, and other pairings defined in
the literature.

Since Hopf algebras play the role of symmetries of noncommutative
spaces, and Hopf-cyclic cohomology extends group and Lie algebra
cohomology~\cite{ConnesMoscovici:HopfCyclicCohomology,
  ConnesMoscovici:HopfCyclicCohomologyIa,
  ConnesMoscovici:HopfCyclicCohomologyII, HKRS:HopfCyclicHomology},
one should expect existence of cup products in Hopf-cyclic cohomology.
There are numerous such products and pairings in the literature
\cite{Gorokhovsky:SecondaryCharacteristicClasses,
  Crainic:CyclicCohomologyOfHopfAlgebras,
  KhalkhaliRangipour:CupProducts, Kaygun:CupProduct,
  Sharygin:CupProducts, Rangipour:CupProductsI} which extend the
characteristic map defined by Connes and
Moscovici~\cite{ConnesMoscovici:HopfCyclicCohomology}.  Let $H$ be a
Hopf algebra, $A$ be a $H$-module algebra and $M$ be an arbitrary
$H$-module/comodule.  In this paper we prove that the pairings and cup
products we enumerated above which extended Connes-Moscovici
characteristic map
\[ HC_{\rm Hopf}^p(H,M)\otimes HC_{\rm Hopf}^0(A,M)\to HC^p(A) \] are
isomorphic natural transformations of isomorphic double functors
defined in their ambient derived categories.  Since our setup uses
module (co)algebras over a fixed base Hopf-algebra, in addition to the
canonical functor $Cyc_\bullet$ which associates an ordinary
(co)algebra a (co)cyclic module, we employ another functor
$C_\bullet$~\cite[Def. 4.7]{Kaygun:BivariantHopf} which is defined
from the category of (co)module (co)algebras in to the category of
(co)cyclic modules.  Because the category of algebras, and therefore
the subcategory $im(C_\bullet)$ of the category of (co)cyclic modules,
is not abelian we will achieve our objective by finding isomorphic
functors on the full double category
$\lmod{\Lambda}\times\lmod{\Lambda}$ whose derivatives on ${\bf
  D}(\lmod{\Lambda})\times{\bf D}(\lmod{\Lambda})$ restricted to ${\bf
  D}(im(C_\bullet))\times {\bf D}(im(C_\bullet))$ yield the pairings
we are interested.  Thus we reduce the task of showing these pairings
are isomorphic in various derived categories, to showing that they
come from isomorphic double functors on the abelian category of
(co)cyclic modules.

Here is a plan of our paper.  In Section~\ref{Relative}, we will
recall few relevant facts about relative (co)homology \`a la
Hochschild~\cite{Hochschild:RelativeHomology}.  In
Section~\ref{DoubleEverything} we develop the necessary machinery for
products of abelian categories and their derived categories.  We use
this machinery in Section~\ref{DiagramModules} to develop a universal
pairing using the double functor $diag_\bullet\Hom_k(\ \cdot\ ,\
\cdot\ )$ we used extensively in \cite{Kaygun:CupProduct} for Connes'
cyclic category $\Lambda$, this time for modules over an arbitrary
small category ${\bf C}$.  By allowing the base category to change,
one can get similar pairing and products for other (co)homology
theories.  In Section~\ref{CyclicCohomology} we prove that the
relative derived category ${\bf D}(\lmod{(\Lambda,{\bf T})})$
implements cyclic (co)homology via cyclic invariant Hochschild
cochains.  We also construct a comparison functor ${\bf
  D}(\lmod{\Lambda})\to {\bf D}(\lmod{(\Lambda,{\bf T})})$ between the
derived category of (co)cyclic modules and the relative derived
category of (co)cyclic modules, which is a homotopy equivalence for a
fixed ground field $k$ of characteristic $0$.  Finally, in
Section~\ref{Pairings} we prove our uniqueness result as we outlined
above.

In this paper we fix a ground field $k$.  We will assume that
$char(k)=0$.  All unadorned tensor products $\otimes$ are taken over
$k$.  All categories we will use are assumed to be small.  All abelian
categories are assumed to have enough injectives and projectives.

\subsection{Acknowledgement}

Most of this article is written during my year-long stay at Ohio State
University, but I finished writing the main argument during my stay at
the Max Planck Institute in Bonn.  I would like to thank both
institutions for their generous support and hospitality.  I thank
Antoine Touze for his help on Lemma~\ref{Agreement}, and Bahram
Rangipour for explaining few key points
in~\cite{KhalkhaliRangipour:CupProducts}.  Last but not the least, I
would like to thank Henri Moscovici for the discussions we had about
this work, and many other things, over our daily coffee breaks.

\section{Relative (Co)Homology}\label{Relative}

In this section we assume $R$ is a unital associative $k$-algebra.
Most of this material can be found in
\cite{Hochschild:RelativeHomology}.

\begin{definition}
  Let $S$ be a subalgebra of $R$.  A morphism of $R$-modules $f\colon
  X\to Y$ is called an $(R,S)$-epimorphism (resp.
  $(R,S)$-monomorphism) if (i) $f$ is an epimorphism
  (resp. monomorphism) of $R$-modules and (ii) $f$ is a split
  epimorphism (resp. monomorphism) of $S$-modules.  A short exact
  sequence $0\to X \xra{v} Y \xra{u} Z \to 0$ of $R$-modules is called
  $(R,S)$-exact if $u$ is an $(R,S)$-epimorphism and $v$ is an
  $(R,S)$-monomorphism.
\end{definition}

\begin{definition}
  An $R$-module $P$ is called an $(R,S)$-projective module if for any
  $(R,S)$-epimorphism $u\colon X\to Y$ and morphism of $R$-modules
  $p\colon P\to Y$ one can find $\tilde{p}\colon P\to X$ such that 
  $\tilde{p}\circ u = p$.
  \begin{equation*}
    \SelectTips{eu}{12}\xymatrix{
    & P \ar[d]^p\ar@{.>}[dl]_{\tilde{p}}\\
    X \ar[r]_u & Y \ar[r] & 0 }
  \end{equation*}
\end{definition}

\begin{definition}
  A $k$-algebra $S$ is called semi-simple if every monomorphism of
  $S$-modules, or equivalently every epimorphism of $S$-modules,
  splits.
\end{definition}

This definition immediately implies the following
\begin{lemma}\label{SemiSimple}
  Let $S$ be a semi-simple subalgebra of $R$.  Then the class of 
  $(R,S)$-epimorphisms and $(R,S)$-monomorphisms coincide with the
  class of ordinary epimorphisms and ordinary monomorphisms 
  of $R$-modules, respectively.
\end{lemma}


\begin{proposition}\label{CutsLikeAKnife}
  Assume $S$ is a semi-simple subalgebra of an algebra $R$.  Then
  ordinary Tor groups ${\rm Tor}^R_*(X,Y)$ and relative Tor groups
  ${\rm Tor}^{(R,S)}_*(X,Y)$ are naturally isomorphic for an arbitrary
  right $R$-module $X$ and an arbitrary left $R$-module $Y$.
  Similarly, ordinary Ext groups $\Ext_R^*(X,Z)$ and the relative Ext
  groups $\Ext_{(R,S)}^*(X,Z)$ are naturally isomorphic for an
  arbitrary pair of $R$-modules $(X,Z)$ of the same parity.
\end{proposition}

\begin{proof}
  Let $X$ be a right $R$-module and let $Y$ be a left $R$-module.
  Since $k$ is a field, the homology of the two sided bar complex
  $\CB_*(X,R,Y)$ which is $\bigoplus_{n\geq 0} X\otimes R^{\otimes
    n}\otimes Y$ with the differentials
  \begin{align*}
    d^\CB_n(x\otimes r_1\otimes\cdots\otimes r_n\otimes y)
     = & (x r_1\otimes r_2\otimes\cdots\otimes r_n\otimes y)\\
       & + \sum_{j=1}^{n-1}(-1)^j (x\otimes\cdots\otimes r_j r_{j+1}
           \otimes\cdots\otimes y)\\
       & + (-1)^n (x\otimes r_1\otimes\cdots\otimes r_{n-1}\otimes r_n y)
  \end{align*}
  is defined for any $n\geq 1$ gives ${\rm Tor}^R_*(X,Y)$.  Using
  \cite[Lem. 2, pg. 248]{Hochschild:RelativeHomology}, we see that for
  any right $R$-module $X$, the module $X\otimes_S R$ is a
  $(R,S)$-projective module.  This implies the relative two sided bar
  complex $\CB_*(X,R|S,Y)$ which is defined as $\bigoplus_{n\geq 0}
  X\otimes_S R^{\otimes_S n}\otimes_S Y$ with the differentials
  induced from the ordinary bar complex yields the relative Tor groups
  ${\rm Tor}^{(R,S)}_*(X,Y)$.  More importantly, there is a comparison
  natural transformation between the derived functors
  \[ c^{X,Y}_*\colon {\rm Tor}^R_*(X,Y)\to {\rm Tor}^{(R,S)}_*(X,Y) \]
  There are similar comparison morphisms between the derived functors
  \[ c_{X,Z}^*\colon \Ext_{(R,S)}^*(X,Z)\to \Ext_R^*(X,Z) \] since
  $\Ext_{(R,S)}^*(X,Z) = H^* \Hom_R(\CB_*(X,R|S,R),Z)$ for an
  arbitrary pair of $R$-modules $(X,Z)$ of the same parity.  If $S$ is
  a semi-simple $k$-algebra then the class of $(R,S)$-projective
  modules coincide with the class of $R$-projective modules because of
  Lemma~\ref{SemiSimple}.  Then the comparison natural transformations
  are isomorphisms.
\end{proof}

\section{Product Categories and Double Functors}\label{DoubleEverything}

\begin{definition}
  Let $\C{U}$ and $\C{V}$ be two $k$-linear small categories.  Define
  a new category $\C{U}\otimes\C{V}$ as follows.  The set of objects
  of $\C{U}\otimes\C{V}$ are pairs of the form $(U,V)$ with $U\in
  Ob(\C{U})$ and $V\in Ob(\C{V})$.  Given two objects $(U,V)$ and
  $(U', V')$ in $Ob(\C{U}\otimes\C{V})$ we define
  \[ \Hom_{\C{U}\otimes\C{V}}((U, V),(U',V')) 
    := \Hom_\C{U}(U,U')\otimes \Hom_\C{V}(V,V')
  \]
  The compositions are defined as $(f\otimes f')\circ(g\otimes g') =
  f\circ g\otimes f'\circ g'$ if $f\otimes f'$ and $g\otimes g'$ are
  two composable morphisms.  Note that we always have
  \[ f\otimes g = (f\otimes t(g))\circ (s(f)\otimes g) = (t(f)\otimes
  g)\circ (f\otimes s(g))\] for any $f\in \Hom_\C{U}$ and
  $g\in\Hom_\C{V}$ where we use $s(\alpha)$ and $t(\alpha)$ to denote
  the source and the target of and morphism $\alpha$.  We also use the
  convention that the symbol for an object also denotes the identity
  morphism on the same object.
\end{definition}

\begin{remark}
  Note that when $\C{U}$ and $\C{V}$ are $k$-linear abelian
  categories, the product category $\C{U}\otimes\C{V}$ is a $k$-linear
  category, but not necessarily an abelian category.  For each object
  $U\in Ob(\C{U})$ and $V\in Ob(\C{V})$ we have subcategories
  $\C{U}\otimes V$ and $U\otimes\C{V}$ of $\C{U}\otimes\C{V}$ which
  consists of objects
  \[ Ob(\C{U}\otimes V) := \left\{(U,V)|\ U\in Ob(\C{U})\right\} 
  \quad\text{  and  }\quad
     Ob(U\otimes\C{V}) := \left\{(U,V)|\ V\in Ob(\C{V})\right\} 
  \]
  where the morphisms are
  \[ \Hom_{\C{U}\otimes V}((U,V),(U',V)) = \Hom_\C{U}(U,U')\otimes k\{id_V\} \]
  and 
  \[ \Hom_{U\otimes\C{V}}((U,V),(U,V')) = k\{id_U\}\otimes \Hom_\C{U}(V,V') \]
  These subcategories are abelian.
\end{remark}

\begin{definition}
  A $k$-linear functor of the form $F\colon\C{U}\otimes\C{V}\to \C{W}$
  is called a $k$-linear double functor.  Such a double functor is
  called exact (resp. left exact or right exact) if $\C{U}$, $\C{V}$
  and $\C{W}$ are $k$-linear abelian categories, and the restriction
  functors
  \[ F_U\colon U\otimes\C{V}\to \C{W}
     \quad\text{ and }\quad
     F_V\colon \C{U}\otimes V\to \C{W}
  \]
  are both exact (resp. left exact or right exact) for any $U\in
  Ob(\C{U})$ and $V\in Ob(\C{V})$.
\end{definition}

\begin{definition}
  Assume $\C{U}$ is an abelian category.  Consider the category
  $\Ch(\C{U})$ of differential $\B{Z}$-graded objects in $\C{U}$ with
  differentials of degree $-1$, and the graded morphisms of
  $\B{Z}$-graded objects between these objects.  This choice of
  morphisms makes each Hom set $\Hom_{\Ch(\C{U})}(U_*,U'_*)$ into a
  differential graded $k$-module as follows: for any $n\in\B{Z}$ we
  let
  \[ \Hom_{\Ch(\C{U})}^n(U_*,U_*') := \prod_{m\in\B{Z}} \Hom_\C{U}(U_m,U'_{m+n})  \]
  The differentials on the Hom modules are given by
  \[ d^n_\C{U}(f_*) := \left(d'_{m+n}\circ f_{m} +
    (-1)^{n+1} f_{m+n-1}\circ d_{m}\right)_{m\in\B{Z}} \] for any
  $f_*\in \Hom_{\Ch(\C{U})}(U_*,U_*')$ where $d_*$ is the differential
  of $U_*$ and $d_*'$ is the differential of $U_*'$.  One can easily
  show that the composition of morphisms is graded and the
  differentials $d^*_{\C{U}}$ satisfy Leibniz rule with respect to
  compositions, i.e.
  \[ d^{i+j}_{\C{U}}(f_*\circ g_*) = d^i_{\C{U}}(f_*)\circ g_* +
  (-1)^i f_*\circ d^j_{\C{U}}(g_*) \] for a composable pair of
  morphisms $f_*$ and $g_*$ of degree $i$ and $j$ respectively.  We
  observe that $ker(d^0_{\C{U}})$ consists of morphisms of
  $\B{Z}$-graded objects of degree $0$ which commute with their
  differentials, and $im(d^1_\C{U})$ consists of morphisms of
  differential graded objects which are null-homotopic.  Thus using
  Hom sets of the form $H_0\Hom^*_{\Ch(\C{U})}(U_*,U_*')$ gives us the
  quotient of the category of differential graded objects and their
  degree 0 morphisms in $\C{U}$ by the subcategory of null-homotopic
  morphisms.  This category is usually denoted by ${\bf K}(\C{U})$.
\end{definition}

\begin{lemma}
  Assume $\C{W}$ has countable products (i.e. limits over countable
  discrete categories), and that such products are exact.  If $F\colon
  \C{U}\otimes\C{V}\to \C{W}$ is a $k$-linear double functor then $F$
  induces a functor of the form ${\bf K}(F)\colon {\bf
    K}(\C{U})\otimes {\bf K}(\C{V}) \to {\bf K}(\C{W})$ which is
  defined on the objects by
  \[ {\bf K}(F)(U_*,V_*):= Tot^\Pi_* F(U_*,V_*) \] for any $U_*\in
  Ob(\Ch(\C{U}))$ and $V_*\in Ob(\Ch(\C{V}))$.
\end{lemma}

\begin{proof}
  First, let us describe $\Ch(\C{U})\otimes\Ch(\C{V})$.  Its objects
  are pairs of the form $(U_*,V_*)$ where $U_*$ and $V_*$ are from
  $\Ch(\C{U})$ and $\Ch(\C{V})$ respectively.  The bi-graded
  $k$-module of morphisms between two objects $(U_*,V_*)$ and
  $(U'_*,V'_*)$ is defined to be
  \begin{align*}
    \Hom_{\Ch(\C{U})\otimes\Ch(\C{V})}^{i,j} & ((U_*,V_*),(U_*',V_*')) 
     := \prod_{p,q\in\B{Z}}
        \Hom_{\C{U}\otimes\C{V}}((U_p,V_q),(U'_{p+i},V'_{q+j}))\\
     := & \prod_{p,q\in\B{Z}}
          \Hom_{\C{U}}(U_p,U'_{p+i})\otimes \Hom_{\C{V}}(V_q,V'_{q+j})
  \end{align*}
  Now observe that $F(U_*,V_*)$ is also a bi-differential object in
  $\C{W}$ with horizontal and vertical differentials for all
  $p,q\in\B{Z}$ are defined as
  \[ d^h_{p,q}:= F(d_p, V_q) \quad\text{ and }\quad d^v_{p,q}:= F(U_p,
  d_q) \] 
  It is easy to see that if $f_*\in \Hom_{\Ch(\C{U})}^i(U_*,U'_*)$ and
  $g_*\in \Hom_{\Ch(\C{V})}^j(V_*,V'_*)$ then
  \[ Tot^\Pi_* F(f_*,g_*) \in \Hom_{\Ch(\C{W})}^{i+j}(Tot^\Pi_*
  F(U_*,V_*), Tot^\Pi_* F(U'_*,V'_*)) \] since $F(f_*,g_*) =
  F(f_*,V'_*)\circ F(U_*,g_*)$ and the composition of morphisms in
  $\Ch(\C{W})$ is graded.  Since the differential on the total complex
  $Tot^\Pi_*F(U_*,V_*)$ is defined as the sum $d^h_{*,*}+d^v_{*,*}$ we
  see that
  \[ d^{i+j}_{\C{W}} Tot^\Pi_m F(f_*,g_*) := \left( (d^h_{p+i,q+j}
    + d^v_{p+i,q+j})F(f_p,g_q) + (-1)^{i+j+1} F(f_p,g_q)(d^v_{p+1,q} +
    d^h_{p,q+1}) \right)_{p+q=m}
  \]
  where $f_*$ and $g_*$ are as before.  One can reduce this formula
  further to
  \begin{align*}
    d^{i+j}_{\C{W}}Tot^\Pi_* F(f_*,g_*) = Tot^\Pi_* F(d^i_{\C{U}} f_*,g_*)
    + (-1)^i Tot^\Pi_* F(f_*, d^j_{\C{V}} g_*)
  \end{align*}
  In order to have a well-defined functor ${\bf K}(F)$, one must have
  the following conditions satisfied.
  \begin{enumerate}[(i)]
  \item If $f_*$ is in $ker(d^0_{\C{U}})$ and $g_*$ is in
    $ker(d^0_{\C{V}})$ then $Tot^\Pi_* F(f_*,g_*)$ is in
    $ker(d_{\C{W}}^0)$. 
  \item For any $a_*\in \Hom_{\Ch(\C{U})}^1(U_*,U'_*)$, $c_*\in
    ker(d^0_{\C{U}})$, $b_*\in ker(d^0_{\C{V}})$ and
    $e_*\in \Hom_{\Ch(\C{V})}^1(V_*,V'_*)$ one must have
    \[ Tot^\Pi_* F(d^1_{\C{U}}(a_*),b_*) + Tot^\Pi_* F(c_*,
    d^1_{\C{V}}(e_*)) \] in the image of $d_{\C{W}}^1$.
  \end{enumerate}
  In order to prove (i) we compute
  \[ d^0_{\C{W}} Tot^\Pi_* F(f_*,g_*)
    =  Tot^\Pi_* F(d^0_{\C{U}} f_*, g_*) +  Tot^\Pi_* F(f_*, d^0_{\C{V}} g_*) = 0
  \]
  For (ii) we observe 
  \begin{align*}
    d^1_{\C{W}} & \left(Tot^\Pi_* F(a_*,b_*) + Tot^\Pi_* F(c_*,e_*)\right) \\
    = & Tot^\Pi_* F(d^1_{\C{U}} a_*, b_*) - Tot^\Pi_* F(a_*, d^0_{\C{V}} b_*)
       + Tot^\Pi_* F(d^0_{\C{U}} c_*, e_*) + Tot^\Pi_* F(c_*, d^1_{\C{V}} e_*)\\
    = & Tot^\Pi_* F(d^1_{\C{U}}(a_*),b_*) + Tot^\Pi_* F(c_*, d^1_{\C{V}}(e_*)) 
  \end{align*}
  as we wanted to prove.
\end{proof}

\begin{remark}
  Given a functor between two small $k$-linear abelian categories
  $G\colon \C{U}\to \C{V}$, we will get a functor of the form
  $\Ch(G)\colon \Ch(\C{U})\to \Ch(\C{V})$.  Both $\Ch(\C{U})$ and
  $\Ch(\C{V})$ are categories enriched over the category of
  differential graded $k$-modules, i.e. they are differential graded
  $k$-linear categories and $\Ch(G)$ is a functor of differential
  graded $k$-linear
  categories~\cite{Keller:OnDifferentialGradedCategories}.  Now, given
  a double functor $F\colon \C{U}\otimes\C{V} \to \C{W}$, one can show
  that $\Ch(F)\colon \Ch(\C{U})\otimes\Ch(\C{V})\to \Ch(\C{W})$ is a
  functor of differential graded $k$-linear categories with a little
  work.  Since the domain is a product category which is normally
  enriched over the category of bi-differential bi-graded $k$-modules,
  the maps on the Hom $k$-modules need to be interpreted carefully
  using total complexes on the Hom $k$-modules.  Then the result we
  prove above is equivalent to the existence of a functor of the form
  \[ H_{p,q}(F)\colon H_p\Ch(\C{U})\otimes H_q\Ch(\C{V}) \to
  H_{p+q}\Ch(\C{W}) \] for $(p,q)=(0,0)$, which actually holds for all
  $(p,q)$.
\end{remark}

\begin{theorem}\label{NaturalTransformation}
  Assume $\C{W}$ has countable products and that such products are
  exact.  If $F\colon \C{U}\otimes\C{V}\to \C{W}$ is a $k$-linear
  exact double functor then ${\bf K}(F)$ extends to a functor on the
  product of the derived categories of the form ${\bf D}(F)\colon {\bf
    D}(\C{U})\otimes {\bf D}(\C{V}) \to {\bf D}(\C{W})$ and we obtain
  natural transformations of double functors
  \[ \Ext_\C{U}^p(\ \cdot\ ,U) \otimes \Ext^q_\C{V}(\ \cdot\ ,V) \to
  \Ext^{p+q}_\C{W}(F(\ \cdot\ ,\ \cdot\ ),F(U,V)) \] for any fixed
  pair of objects $U\in Ob(\C{U})$ and $V\in Ob(\C{V})$, compatible
  with the triangulated structures in both variables for any $p$ and
  $q$.
\end{theorem}

\begin{proof}
  We will show that quasi-isomorphisms in each variable, with the
  other variable fixed, are sent to quasi-isomorphisms in $\Ch(\C{W})$
  in order to get the required extension.  We will prove this for the
  first variable.  The proof for the second variable is identical.
  Assume $f_*\colon U_*\to U'_*$ is a quasi-isomorphism in
  $\Ch(\C{U})$.  By \cite[Corollary 1.5.4]{Weibel:HomologicalAlgebra}
  $f_*$ is a quasi-isomorphism if and only if the mapping cone of
  $f_*$, here denoted by ${\rm Cone(f_*)}$, is acyclic and we have a
  split exact sequence of differential graded objects
  \[ 0 \to U'_* \to {\rm Cone}(f_*) \to U_*[-1] \to 0 \] Since
  $Tot^\Pi_*$ is an exact functor from $\Ch^2(\C{W})$ the category of
  bi-differential bi-graded objects in $\C{W}$ to $\Ch(\C{W})$, and
  $F$ is exact in the first variable, this short exact sequence
  translates to
  \[ 0 \to Tot^\Pi_* F(U'_*,V_*) \to Tot^\Pi_* F({\rm Cone}(f_*),V_*)
  \to Tot^\Pi_* F(U_*[-1],V_*) \to 0 \] Again, since $F$ is exact,
  $F({\rm Cone}(f_*),V_q)$ is acyclic for every $q\in\B{Z}$.  Then by
  a spectral sequence argument we conclude that the total complex
  $Tot^\Pi_* F({\rm Cone}(f_*),V_*)$ is acyclic.
  Therefore $Tot^\Pi_* F(f_*,V_*)$ induces an isomorphism on homology
  for any $V_*$ in $\Ch(\C{V})$ using the resulting long exact
  sequence in homology.  This finishes the proof of the first part.
  The subcategories that we are interested in the derived categories
  ${\bf D}(\C{U})$, ${\bf D}(\C{V})$ and ${\bf D}(\C{W})$ are the full
  subcategories of the complexes whose homologies are concentrated at
  only one degree.  Since $F$ is exact, ${\bf D}(F)$ sends the objects
  of the form $U[i]$ in ${\bf D}(\C{U})$ and $V[j]$ in ${\bf
    D}(\C{V})$ to an object of the form $F(U,V)[i+j]$ in ${\bf
    D}(\C{W})$.  Here $U$, $V$ and $F(U,V)$ are ordinary objects in
  $\C{U}$, $\C{V}$ and $\C{W}$ respectively, and we use the notation
  that for an ordinary object $A$ in an abelian category $\C{A}$ the
  symbol $A[n]$ represents an object in the derived category of
  $\C{A}$ which is the same $A$ considered as a complex with 0
  differentials and concentrated at degree $n\in\B{Z}$.  The result
  follows.
\end{proof}

\begin{remark}
  The natural transformation of double functors we obtained in
  Theorem~\ref{NaturalTransformation} is actually a natural
  transformation of quadruple functors if we do not fix two of the
  variables. But we will need the result in this form.
\end{remark}

\begin{lemma}\label{Agreement}
  Assume $G^*\colon \C{U} \otimes \C{V}\to \lmod{k}$ is a graded
  double functor which is a cohomological $\delta$-functor
  \cite[Def. 2.1.1]{Weibel:HomologicalAlgebra} in each variable. Fix
  $U\in Ob(\C{U})$ and $V\in Ob(\C{V})$.  Assume also that we have
  natural transformations of double functors of the form
  \[ \eta^{p,q}\colon 
     \Ext^p_{\C{U}}(\ \cdot\ ,U)\otimes\Ext^q_{\C{V}}(\ \cdot\ ,V) 
     \to G^{p+q}(\ \cdot\ ,\ \cdot\ )
  \]
  compatible with the $\delta$-structures in both of the variables for
  any $p,q\geq 0$.  If there exists an $n\in\B{N}$ such that
  $\eta^{p,q}=0$ for any $p+q=n$ then $\eta^{p,q}=0$ for any $p+q\geq
  n$.
\end{lemma}

\begin{proof}
  Fix another object $V'\in Ob(\C{V})$ and consider a short exact
  sequence $0\xla{} U''\xla{}P\xla{} U'\xla{}0$ in $\C{U}$ where $P$
  is projective.  We obtain a commutative diagram of the form\small
  \[\SelectTips{eu}{12}\xymatrix{
    \Ext_{\C{U}}^p(U',U)\otimes\Ext_{\C{V}}^q(V',V)
    \ar[d]_{\eta^{p,q}}^0 \ar[r]^-{\delta^p\otimes id\ } &
    \Ext_{\C{U}}^{p+1}(U'',U)\otimes\Ext_{\C{V}}^q(V',V)
    \ar[d]_{\eta^{p+1,q}} \ar[r]^0 &
    \Ext_{\C{U}}^{p+1}(P,U)\otimes\Ext_{\C{V}}^q(V',V)
    \ar[d]^{\eta^{p+1,q}} \\
    G^{p+q}(U',V') \ar[r] & G^{p+q+1}(U'',V') \ar[r] & G^{p+q+1}(P,V')
  }\] \normalsize Since $P$ is projective $\Ext_{\C{U}}^{p+1}(P,U) =
  0$ and $\delta^p\otimes id$ is an epimorphism for any $p\geq 0$.  On
  the other hand $\eta^{p+1,q}(\delta^p\otimes id) = 0$ since
  $\eta^{p,q}=0$.  Because $\delta^p\otimes id$ is an epimorphism, we
  have $\eta^{p+1,q}=0$.  One can repeat the same argument for the
  other variable.  Since $U''$ was arbitrary and short exact sequences
  of the form $0\xla{} U''\xla{}P\xla{} U'\xla{}0$ always exist
  because $\C{U}$ (and $\C{V}$ for the second variable) have enough
  projectives, the result follows.
\end{proof}

\section{Diagram Modules and Pairings}\label{DiagramModules}

\begin{definition}
  Let ${\bf C}$ be a small category and let $\lmod{{\bf C}}$ (resp.
  $\rmod{{\bf C}}$) denote the category of covariant (resp.
  contravariant) functors from ${\bf C}$ to $\lmod{k}$ and their
  natural transformations.  We will call such functors left
  (resp. right) ${\bf C}$-modules.  Such a module $X_\bullet$ is a
  direct sum of $k$-modules indexed by the set of objects of ${\bf C}$
  of the form $\bigoplus_{a\in Ob({\bf C})} X_a$.  The primary
  examples we have in mind are ${\bf C}=\Delta$ the simplicial
  category, or ${\bf C} =\Lambda$ Connes' cyclic category.  We will
  use the following notation
  \[ (a\xla{f}b)\tr x := X_f(x) \qquad \text{(resp. } 
      x\tl (a\xla{f}b) := X_f(x) \text{)} \] for any $x\in X_a$
  (resp. $x\in X_b$) where $X_f\colon X(b)\to X(a)$ is the evaluation
  of $X_\bullet$ on the morphism $a\xla{f}b$ in ${\bf C}$.  
\end{definition}

\begin{remark}
  In order to simplify notation, for a small category ${\bf C}$ we
  will use the notation $\Hom_{\bf C}$ for morphisms of ${\bf
    C}$-modules.  We will extend this simplification to the derived
  functors as well and use $\Ext_{\bf C}^*$ for the derived functors
  of the double functor $\Hom_{\bf C}$ for both left and right ${\bf
    C}$-modules.
\end{remark}

\begin{definition}
  We define the ${\bf C}$-module $k_\bullet$ by letting $k_a = k$ for
  any $a\in Ob({\bf C})$ and we let
  \[ 1_a\tl(a\xla{f}b) = 1_b \quad\text{ or }\quad
     (a\xla{f}b)\tr 1_b = 1_a 
  \]
  for any $f\colon b\to a$ in ${\bf C}$ depending on whether we view
  it as a left or right ${\bf C}$-module.
\end{definition}

\begin{definition}
  Assume $X_\bullet$ is a left ${\bf C}$-module and $Y_\bullet$ is a
  right ${\bf C}$-module.  Let
  \[ diag_\bullet \Hom_k(X_\bullet,Y_\bullet) 
    := \bigoplus_{a\in Ob({\bf C})} \Hom_k(X_a,Y_a)
  \] 
  By definition $diag_\bullet \Hom_k(X_\bullet,Y_\bullet)$ is a
  $k$-module indexed by the set of objects of ${\bf C}$.  Also, given
  any $\psi\in diag_a\Hom_k(X_\bullet,Y_\bullet)$ and $f\in\Hom_{\bf
    C}(b,a)$ define
  \[ (\psi\cdot(a\xla{f}b))(x) := \psi((a\xla{f}b)\tr x)\tl(a\xla{f}b) 
     \in diag_b\Hom_k(X_\bullet,Y_\bullet)
  \]
  for any $x\in X_b$.  So, $\psi\cdot(a\xla{f}b)$ is in 
  $diag_b\Hom_k(X_\bullet,Y_\bullet)$
\end{definition}

\begin{proposition}\label{DiagonalHom}
  The assignment $diag_\bullet\Hom_k(\ \cdot\ ,\ \cdot\ )$ defines
  an exact double functor of the form
  \[ diag_\bullet\Hom_k(\ \cdot\ ,\ \cdot\ )\colon
    (\rmod{{\bf C}})^{op}\otimes \lmod{{\bf C}}\to \lmod{{\bf C}} \]
  which induces a natural transformation of double functors
  \[ (\ \cdot\ \smile \cdot\ )\colon
     \Ext_{\bf C}^p(k_\bullet,\ \cdot\ )\otimes
     \Ext_{\bf C}^q(\ \cdot\ , k_\bullet)
     \to \Ext_{\bf C}^{p+q}(diag_\bullet\Hom_k(\ \cdot\ ,\ \cdot\ ), k_\bullet)
  \]
  for any $p,q\geq 0$.
\end{proposition}

\begin{proof}
  First, we must show that the action of ${\bf C}$ on $diag_\bullet
  \Hom_k(X_\bullet,Y_\bullet)$ is associative for any left ${\bf
    C}$-module $X_\bullet$ and right ${\bf C}$-module $Y_\bullet$.  We
  observe
  \begin{align*}
  ((\psi\cdot(a\xla{f}b))\cdot(b\xla{g}c))(x)
     = & (\psi\cdot(a\xla{f}b))((b\xla{g}c)\tr x)\tl(b\xla{g}c)\\
     = & (\psi((a\xla{f}b)\tr((b\xla{g}c)\tr x))\tl(a\xla{f}b))\tl(b\xla{g}c)\\
     = & \psi((a\xla{fg}c)\tr x)\tl(a\xla{fg}c)\\
     = &  (\psi\cdot(a\xla{fg}c))(x)
  \end{align*} 
  for any $x\in diag_c\Hom_k(X_\bullet,Y_\bullet)$ and $f\colon b\to
  a$ and $g\colon c\to b$ in ${\bf C}$.  The exactness of the double
  functor follows form the fact that $k$ is a field.  Now, observe
  that $diag_\bullet\Hom_k(k_\bullet,k_\bullet) \cong k_\bullet$ apply
  Theorem~\ref{NaturalTransformation} to get the prescribed natural
  transformation.
\end{proof}

\begin{remark}
  One has to be careful in interpreting the derivatives of the double
  functor and the pairing we obtained in Proposition~\ref{DiagonalHom}
  in cases where one would like to use bounded above or bounded below
  derived categories.  In such, cases we have
  \[ {\bf D}_+(diag_\bullet\Hom_k)\colon {\bf D}_-(\lmod{{\bf C}})
            \otimes {\bf D}_+(\rmod{{\bf C}})
     \to {\bf D}_+(\rmod{{\bf C}})
  \]
  or 
  \[ {\bf D}_-(diag_\bullet\Hom_k)\colon {\bf D}_+(\lmod{{\bf C}})
            \otimes {\bf D}_-(\rmod{{\bf C}})
     \to {\bf D}_-(\rmod{{\bf C}})
  \]
 because our functor $diag_\bullet\Hom_k$ is contravariant in
 the first variable.
\end{remark}

\begin{remark}
  For the curious reader who would like to see an explicit formula for
  the pairing we defined above, we note that the pairing is a slight
  modification of the external product in cohomology.  So, the
  Alexander-Whitney map~(cf. \cite[Thm. 8.5]{MacLane:Homology} or
  \cite[8.5.4]{Weibel:HomologicalAlgebra}), applied correctly is going
  to work.  We ask the reader to pick his/her favorite cosimplicial
  module $\C{X}_{\bullet,\bullet}$ which consists of injective ${\bf
    C}$-modules whose (singular) homology is the ${\bf C}$-module
  $X_\bullet$ concentrated at degree 0, and simplicial module
  $\C{Y}_{\bullet,\bullet}$ which consists of projective ${\bf
    C}$-modules whose (singular) homology is the ${\bf C}$-module
  $Y_\bullet$ concentrated at degree 0.  Such modules exist because of
  Dold-Kan equivalence~\cite{DoldPuppe:Correspondence}.  Then for two
  given cochains $\xi\colon k_\bullet\to \C{X}_{p.\bullet}$ and
  $\nu\colon \C{Y}_{q,\bullet}\to k_\bullet$ we define a new cochain
  $\xi\smile\nu\colon
  diag_\bullet\Hom_k(\C{X}_{p+q,\bullet},\C{Y}_{p+q,\bullet})\to
  k_\bullet$ by
  \[ (\xi\smile \nu)(\eta) := \nu\circ
  \underbrace{\partial^\C{Y}_{q+2}\circ\cdots\partial^\C{Y}_{p+q+1}}_{p\text{-terms}}
      \circ\eta\circ
      \underbrace{\partial^\C{X}_{p+q}\circ\cdots\circ\partial^\C{X}_{p+1}}_{q\text{-terms}}
      \circ\xi
  \]
  for any $\eta\in
  diag_\bullet\Hom_k(\C{X}_{p+q,\bullet},\C{Y}_{p+q,\bullet})$ where
  we use $\partial^\C{Z}_i$ to denote the (co)face maps of a
  (co)simplicial object $\C{Z}_\bullet$.
\end{remark}

\section{Cyclic (Co)Homology}\label{CyclicCohomology}

\begin{definition}
  Let $\Lambda$ also denote the $k$-algebra generated by the arrows of
  Connes' cyclic category $\Lambda$.  Here we will give a specific
  presentation of the $k$-algebra $\Lambda$.  We will denote the
  generators by $\partial^n_j\colon [n]\to [n+1]$, $\sigma^n_i\colon
  [n+1]\to [n]$ and $\tau_n^\ell\colon [n]\to [n]$ for any $[n]\in
  Ob(\Lambda)$ with $0\leq j\leq n+1$, $0\leq i\leq n$ and $0\leq
  \ell\leq n$.  The relations are
  \begin{align*}
    \partial^{n+1}_i\partial^n_j & = \partial^{n+1}_{j+1}\partial^n_i \text{ for } i\leq j 
    \qquad\qquad
    \sigma^n_j\sigma^{n+1}_i  = \sigma^n_i\sigma^{n+1}_{j+1} \text{ for } i\leq j\\
    \partial^n_i\sigma^n_j & = \sigma^{n+1}_{j+1}\partial^{n+1}_i \text{ for } i\leq j 
    \qquad\qquad
    \partial^n_i\sigma^n_j  = \sigma^{n+1}_j\partial^{n+1}_{i+1} \text{ for } i>j\\
    \sigma^n_i\partial^n_i & = id_i = \sigma^n_i\partial^n_{i+1} 
    \qquad\qquad\qquad\quad  \tau_n^\ell\tau_n = \tau_n^{\ell+1} \text{ and } \tau_n^{n+1} = id_n \\
    \partial^n_j\tau_n^i & = \tau_{n+1}^{i+p}\partial^n_q \text{ for }  i + j = (n+1)p + q \qquad
    \sigma^n_j\tau_{n+1}^i = \tau_n^{i-p}\sigma^n_q \text{ for } i+j = (n+1)p + q
  \end{align*}
  All other products between the generators are $0$.  Note that
  $\Lambda$ is not a unital $k$-algebra, but a $H$-unital
  algebra~\cite{Wodzicki:Excision}.  One also can view $\Lambda$ first
  as a bimodule then as an algebra over its subalgebra $\C{K}:=
  \bigoplus_{n\geq 0} k\{id_n\}$ generated by $id_n$ for any $n\geq
  0$.
\end{definition}

\begin{remark}\label{Lift}
  Let $\B{F}$ be the set of all finite subsets of $\B{N}$ the set of
  all natural numbers.  Note that $\Lambda$ is the colimit of unital
  $k$-algebras $\colim_{U\in\B{F}}\Lambda_U$ where for an arbitrary
  $U\in\B{F}$ the $k$-algebra $\Lambda_U$ is the unital subalgebra of
  $\Lambda$ generated by elements $\Psi$ which satisfy the property
  that $\Psi = id_m \Psi id_n$ for some $m,n\in U$.  Thus $\Lambda$ is
  $H$-unital.  Because of this property, we are interested only in
  {\em locally finite and faithful modules} over $\Lambda$.  These are
  graded $k$-modules $M_* = \bigoplus_{n\geq 0} M_n$ such that $id_n
  m_n = m_n$ for any $m_n\in M_n$.  These modules have the property
  that they can be written as a colimit of unital modules over the
  unital algebras $\Lambda_U$ for $U\in\B{F}$.  Thus one can prove
  statements for the unital algebras $\Lambda_U$ for $U\in\B{F}$ and
  their modules $M_U = {\rm Res}^{\Lambda}_{\Lambda_U} M$ then lift
  the argument to $\Lambda$ and $M$ by using a colimit taken over
  $U\in\B{F}$.
\end{remark}

\begin{definition}
  Let $\C{X} = \bigoplus_{n,m\in\B{N}} X_{n,m}$ be a $\C{K}$-bimodule.
  Assume $\C{A}$ is a $\C{K}$-algebra.  An isomorphism of
  $\C{K}$-bimodules $\omega_\C{X}\colon\C{X}\otimes_\C{K}\C{A} \to
  \C{A}\otimes_\C{K}\C{X}$ is called a transposition if (i) one has a
  commutative diagram of the form
  \begin{equation*}
  \SelectTips{eu}{12}\xymatrix{
  \C{A}\otimes_\C{K}\C{A}\otimes_\C{K}\C{X}
  \ar[d]_{\mu\otimes\C{X}}
  \ar[r]^{\C{A}\otimes\omega_\C{X}} & 
  \C{A}\otimes_\C{K}\C{X}\otimes_\C{K}\C{A}
  \ar[r]^{\omega_\C{X}\otimes\C{A}} &
  \C{X}\otimes_\C{K}\C{A}\otimes_\C{K}\C{A}
  \ar[d]^{\C{X}\otimes\mu}\\
  \C{A}\otimes_\C{K}\C{X} \ar[rr]_{\omega_\C{X}} & & \C{X}\otimes_\C{K}\C{A}
  }
  \end{equation*}
  where $\mu\colon\C{A}\otimes_\C{K}\C{A}\to\C{A}$ is the
  multiplication structure on $\C{A}$, and (ii) we have
  $\omega(1_p\otimes x) = (x\otimes 1_q)$ for any $x\in X_{p,q}$ and
  $p,q\in\B{N}$. 
\end{definition}
  
\begin{remark}\label{Standard}
  Consider $\Delta$ the subalgebra of $\Lambda$ generated by elements
  $\partial^n_j$, $\sigma^n_i$ and $id_n$ for any $n\in\B{N}$ and all
  possible $i,j$; and the subalgebra ${\bf T}$ of $\Lambda$ generated
  by $\tau_n^i$ for all possible $n\geq 0$ and $0\leq i\leq n$.
  Define a transposition $\omega\colon \Delta\otimes_\C{K}{\bf T}\to
  {\bf T}\otimes_\C{K} \Delta$ using the relations in $\Lambda$
  \begin{align*}
    \omega(\partial^n_j\otimes\tau_n^i) 
     = & \tau_{n+1}^{i+p}\otimes \partial^n_q\text{ with } (i+j) = (n+1)p + q \text{ and } 0\leq q\leq n\\
    \omega(\sigma^n_s\otimes\tau_{n+1}^s) 
     = & \tau_n^{s-a}\otimes\sigma^n_b \text{ with } (s+t) = (n+2)a + b \text{ and } 0\leq b \leq n+1
  \end{align*}
  It is easy to see that $\omega$ is invertible and the inverse is
  given by
  \begin{align*}
    \omega^{-1}(\tau_{n+1}^i\otimes\partial^n_j)
     =  & \partial^n_q\otimes\tau_n^{i-p} \text{ with } (-i+j) = (n+2)(-p) + q \text{ and } 0\leq q\leq n+1\\
    \omega^{-1}(\tau_n^s\otimes \sigma^n_t)
     =  & \sigma^n_b\otimes \tau_{n+1}^{s-a} \text{ with } (-s+t) = (n+1)(-a) + b \text{ and } 0\leq b\leq n
  \end{align*}
  It is tedious but easy to show that $\omega$ and $\omega^{-1}$ are
  transpositions.
\end{remark}

\begin{remark}
  We can even split the subalgebra $\Delta$ of $\Lambda$ into 2 pieces
  using an appropriate distributivity law as follows.  Let $\C{F}$ be
  the subalgebra of $\Lambda$ generated by $\partial^n_i$ and $id_n$ for
  all $n\geq 0$ and $0\leq i\leq n+1$.  Let $\C{D}$ be the subalgebra
  of $\Lambda$ generated by all $\sigma^n_i$ and $id_n$ for all $n\geq 0$
  and $0\leq i\leq n$.  There is a distributivity law of the form
  $\zeta\colon \C{F}\otimes_\C{K}\C{D}\to \C{D}\otimes_\C{K}\C{F}$
  coming from the relations in $\Lambda$.  Note that this distributivity
  law is not an isomorphism because of the relation
  $\sigma^n_i\partial^n_i = id_n = \sigma^n_i\partial^n_{i+1}$ for all
  $0\leq i\leq n$ and $n\geq 0$.
\end{remark}

\begin{definition}
  Let $\CH_n$ be the left ideal of $\Delta$ generated by $id_n$.
  Define $d^\CH_n\colon \CH_{n+1} \to \CH_n$ by using the elements
  \[ d^\CH_n = \sum_{j=0}^{n+1} (-1)^j \partial^n_j \] via right
  multiplication for any $n\geq 0$.  One can easily see that
  $d^\CH_{n+1}d^\CH_n = 0$ for any $n\geq 0$.
\end{definition}

\begin{lemma}
  The differential graded $\Delta$-module $(\CH_*,d^\CH_*)$ is a
  $\Delta$-projective resolution of the trivial left $\Delta$-module
  $k_\bullet$.
\end{lemma}

\begin{proof}
  The proof we present here is the same as that of
  \cite[Lem. 2]{Connes:ExtFunctors}.  Notice that the left ideal
  $\left<id_n\right|$ of $\Delta$ generated by $id_n$ is the free
  $k$-module on the set $\bigsqcup_m \Hom_\Delta(n,m)$ where, by abuse
  of notation, $\Delta$ denotes the category of finite ordinals and
  their order preserving maps.  We observe that the arrows of this
  category generates our algebra $\Delta$ over $k$ and the action of
  the algebra $\Delta$ is defined by pre-compositions on this set of
  generators.  In other words, for each fixed $m$ the differential
  graded $k$-module $id_m\CH_*$ computes the simplicial homotopy of the
  simplicial $k$-module $k[\Delta^m]:= k[\Hom_{\Delta}(\bullet,m)]$
  which is 0 everywhere except at degree $0$, and is the ground field
  at degree $0$.
\end{proof}

\begin{definition}
  Let $\C{X}$ be a $\C{K}$-bimodule and $\C{A}$ be an augmented
  $\C{K}$-algebra with augmentation $\epsilon\colon \C{A}\to \C{K}$.
  Let $\omega_\C{X}\colon \C{A}\otimes_\C{K}\C{X}\to
  \C{X}\otimes_\C{K}\C{A}$ be a transposition.  Then $\C{X}$ carries a left
  $\C{A}$-module structure $\lambda_\C{X}$ which is defined as
  $(\C{X}\otimes_\C{K}\epsilon)\omega_\C{X}$.
\end{definition}
  
\begin{definition}\label{Diagonal}
  Let $\C{X}$ and $\omega_\C{X}$ be as before.  Let $\C{Y} =
  \bigoplus_{n,m\in\B{N}} Y_{n,m}$ be another $\C{K}$-bimodule and
  assume we have another transposition $\omega_\C{Y}\colon
  \C{A}\otimes_\C{K} \C{Y}\to \C{Y}\otimes_\C{K}\C{A}$.  Then the
  product $\C{X}\otimes_\C{K}\C{Y}$ carries a left $\C{A}$-module
  structure which is denoted by $\C{X}\odot\C{Y}$.  The $\C{A}$-module
  structure on $\C{X}\odot\C{Y}$ comes from the product transposition
  $\omega_{\C{X}\odot\C{Y}}\colon
  \C{A}\otimes_\C{K}\C{X}\otimes_\C{K}\C{Y}\to\C{X}\otimes_\C{K}\C{Y}
  \otimes_\C{K}\C{A}$ and the augmentation
  $\epsilon\colon\C{A}\to\C{K}$.  The product transposition is defined
  as
  \[ \omega_{\C{X}\odot\C{Y}} := 
      (\C{X}\otimes \omega_\C{Y})\circ(\omega_\C{X}\otimes\C{Y})
  \]
  and we let the left $\C{A}$-module structure
  $\lambda_{\C{X}\odot\C{Y}}\colon
  \C{A}\otimes_\C{K}(\C{X}\odot\C{Y}) \to\C{X}\odot\C{Y}$ by
  \[ \lambda_{\C{X}\odot\C{Y}} := 
     ((\C{X}\odot\C{Y})\otimes\epsilon)\circ\omega_{\C{X}\odot\C{Y}}
  \]
\end{definition}

\begin{remark}
  Any ${\bf T}$-module $X_\bullet := \bigoplus_{n\geq 0} X_n$ admits a
  canonical transposition $\omega_X\colon {\bf T}\otimes_\C{K}
  X_\bullet\to X_\bullet\otimes_\C{K}{\bf T}$ which is defined as
  \[ \omega_X(\tau_n^\ell\otimes x) = \tau_n^\ell\cdot
  x\otimes\tau_n^\ell \] for every $n\in \B{N}$, $x\in X_n$
  and $\ell\in\B{Z}$.
\end{remark}

\begin{lemma}
  Let $X_\bullet$ and $Y_\bullet$ be two ${\bf T}$-modules.  Then one
  has an isomorphism of $k$-modules of the form 
  \[ X_\bullet\otimes_{\bf T} Y_\bullet \cong
  k_\bullet\otimes_{\bf T}(X_\bullet\odot Y_\bullet) \]
\end{lemma}

\begin{proof}
  We are going to view $X_\bullet = \bigoplus_{n\in\B{N}} X_n$ as a
  left ${\bf T}$-module via the action  $\tau_n\cdot x := x\cdot\tau_n^{-\ell}$
   for any $x\in X_n$ and $\ell\in\B{Z}$.  Since 
   \[ X_\bullet\otimes_{\bf T} Y_\bullet = \bigoplus_{n\in\B{N}}
   X_n\otimes_{\B{Z}/(n+1)} Y_n \] our statement reduces to proving $
   X\otimes_G Y \cong k\otimes_G(X\odot Y) $ where $G$ is a finite
   abelian group, $X$ is a right $G$-module, $Y$ is a left $G$-module
   and $X\odot Y$ is the diagonal $G$-module $g\cdot(x\otimes y) =
   g\cdot x\otimes g\cdot y := x\cdot g^{-1}\otimes g\cdot y$ with
   $g\in G$, $x\in X$ and $y\in Y$.
\end{proof}

\begin{theorem}\label{Equivalent}
  Let $X_\bullet$ be a right $\Lambda$-module (i.e. a cyclic module) and
  $Y_\bullet$ be a left $\Lambda$-module (i.e. a cocyclic module).  Let
  $\CH_*^\lambda(X_\bullet)$ be the cyclic co-invariant quotient
  complex of the Hochschild complex of $X_\bullet$ and
  $\CH_\lambda^*(Y_\bullet)$ be the cyclic invariant subcomplex of the
  Hochschild complex of $Y_\bullet$.  Then
  \[ {\rm Tor}^{(\Lambda,{\bf T})}_*(X_\bullet,k_\bullet) \cong
     HC^\lambda_*(X_\bullet) \quad\text{ and }\quad 
     \Ext_{(\Lambda,{\bf T})}^*(k_\bullet,Y_\bullet) \cong
     HC_\lambda^*(Y_\bullet) 
     \] 
   where $HC^\lambda_*(X_\bullet)$ and $HC_\lambda^*(Y_\bullet)$ are
   the homologies of the complexes $\CH_*^\lambda(X_\bullet)$ and
   $\CH^*_\lambda(Y_\bullet)$, respectively.
 \end{theorem}
 
\begin{proof}
  Observe that we have a basis for $\Lambda$ which consists of elements of the form
  \[ \sigma^m_{i_m}\cdots\sigma^n_{i_n} 
     \partial^n_{j_n}\cdots\partial^\ell_{j_\ell}\tau_\ell^a
     \quad\text{ where } i_m<\cdots<i_n \text{ and } j_n>\cdots>j_\ell 
  \]
  Using the transpositions $\omega$ and $\omega^{-1}$ we defined in
  Remark~\ref{Standard} we see that
  \[ \CB_n(\Lambda,\Lambda|{\bf T},k_\bullet)
      = \overbrace{\Lambda\otimes_{\bf T}\cdots
        \otimes_{\bf T}\Lambda}^{n+1\text{-times}}\otimes_{\bf T}k_{\bullet}
      \cong \overbrace{\Delta\odot\cdots
        \odot\Delta}^{n+1\text{-times}}\odot k_\bullet
      = \CB_n(\Delta,\Delta,k_\bullet)
  \]
  The left $\Delta$-module structure on $\Delta^{\odot n+1}\odot
  k_\bullet$ comes from the left regular representation of $\Delta$ on
  itself on the left-most tensor component.  The ${\bf T}$-module
  structure comes from the diagonal ${\bf T}$-module structure as
  defined in Definition~\ref{Diagonal} coming from the transposition
  $\omega^{-1}$ defined in Remark~\ref{Standard}.  Since ${\bf T}$
  is semi-simple, the resolution $\CB_*(\Delta,\Delta,k_\bullet)$ can be
  replaced by the differential $\Delta$-module $\CH_*$ which is also a
  left $\Lambda$-module structure coming from the transposition
  $\omega^{-1}$.  Then, the two sided bar complex
  $\CB_*(X_\bullet,\Lambda|{\bf T},k_\bullet)$ can be replaced by
  \[ X_\bullet\otimes_\Lambda \CH_* \cong 
     k_\bullet\otimes_{\bf T}\left(X_\bullet\otimes_\Delta\CH_*\right)
     = \CH^\lambda_*(X_\bullet) 
  \] 
  The proof for $\Ext_{(\Lambda,{\bf T})}^*(k_\bullet,Y_\bullet)$ is
  similar.
\end{proof}

\begin{proposition}\label{Comparison}
  We have the natural isomorphisms of derived functors
  \[ c_*^{\ \cdot\ ,\ \cdot\ }\colon {\rm Tor}^\Lambda_*(\ \cdot\ ,\ \cdot\ )
     \to {\rm Tor}^{(\Lambda,{\bf T})}_*(\ \cdot\ ,\ \cdot\ )
     \qquad
     c^*_{\ \cdot\ ,\ \cdot\ }\colon \Ext_{(\Lambda,{\bf T})}^*(\ \cdot\ ,\ \cdot\ )
     \to \Ext_{\Lambda}^*(\ \cdot\ ,\ \cdot\ )
  \]
\end{proposition}

\begin{proof}
  We observe that ${\bf T}$ is a semi-simple subalgebra of $\Lambda$
  since we assume $char(k)=0$ throughout.  Now we use
  Proposition~\ref{CutsLikeAKnife} and Remark~\ref{Lift}.
\end{proof}

\section{Pairings in Cyclic (Co)Homology}\label{Pairings}

In this section we will use the notation and the terminology of
\cite{Kaygun:CupProduct} and \cite{Kaygun:BivariantHopf} with the
simplification that we use the same $k_\bullet$ for the trivial left
and right $\Lambda$-module.  In particular, $C_\bullet(Z,M)$ will
denote the (co)cyclic module (i.e. $\Lambda$-module) associated with a
$H$-module (co)algebra $Z$ with coefficients in an arbitrary
$H$-module/comodule $M$.  Since the category of $\Lambda$-modules is
abelian, Hopf-cyclic (co)homology of (co)cyclic modules are specific
derived functors on this category
\[  HC_{\rm Hopf}^*(A,M) = \Ext_\Lambda^*(C_\bullet(A,M),k_\bullet) 
    \quad\text{ and }\quad
    HC_{\rm Hopf}^*(C,M) = \Ext_\Lambda^*(k_\bullet,C_\bullet(C,M)) 
\]
for an arbitrary $H$-module algebra $A$ and $H$-module coalgebra $C$.

We recall the following definition from
\cite[Def. 2.2]{Kaygun:CupProduct} to fix notation: $C$ is said to act
on $A$ if there is a morphism of $k$-modules $\triangleright\colon
C\otimes A\to A$ which satisfies (i) $c\triangleright(a_1a_2)=
(c^{(1)}\triangleright a_1)(c^{(2)}\triangleright a_2)$ and (ii) $c\tr
1_A = \varepsilon(c) 1_A$ for any $a_1,a_2\in A$ and $c\in C$.  The
action is called $H$-equivariant if $h(c\tr a) = h(c)\tr a$ for any
$h\in H$, $a\in A$ and $c\in C$ where we use $h(c)$ to denote the
action of $H$ on the module (co)algebra $C$.

We obtained the following result in
\cite[Prop. 2.7]{Kaygun:CupProduct}.
\begin{lemma}\label{ModulePairing}
  Assume $C$ acts on $A$ equivariantly.  The morphism of graded
  $k$-modules
  \[\phi_\bullet \colon Cyc_\bullet(A) \to diag_\bullet\Hom_k(C_\bullet(C,M),C_\bullet(A,M))
  \] defined for $a_0\otimes\cdots\otimes a_n\otimes m\in C_n(A,M)$ and
  $c_0\otimes\cdots\otimes c_n\otimes m\in C_n(C,M)$ for any $n\geq 0$
  by
  \[ \phi_n(a_0\otimes\cdots\otimes a_n)(c_0\otimes\cdots\otimes c_n\otimes m)
     = c_0\tr a_0\otimes\cdots\otimes c_n\tr a_n\otimes m
  \]
  is a morphism of cyclic modules.  
\end{lemma}

\begin{theorem}\label{CohomologyPairing}
  The morphism of cyclic modules $\phi_\bullet$ we defined in
  Lemma~\ref{ModulePairing} induces a pairing of the form
  \[ (\ \cdot\ \smile\ \cdot\ ) \colon HC^p_{\rm
    Hopf}(C,M)\otimes HC^q_{\rm Hopf}(A,M)\to HC^{p+q}(A) \] for any
  $p,q\geq 0$ where we use $HC^*_{\rm Hopf}$ to denote Hopf-cyclic
  cohomology and $HC^*$ to denote the ordinary cyclic cohomology
  functors.
\end{theorem}

\begin{proof}
  The pairing comes from Proposition~\ref{DiagonalHom} followed by
  Lemma~\ref{ModulePairing}.  
\end{proof}

\begin{remark}
  Theorem~\ref{CohomologyPairing} gives us a pairing defined in ${\bf
    D}(\lmod{\Lambda})$ the derived category of $\Lambda$-modules.
  Note that the pairing we obtain in Proposition~\ref{DiagonalHom} for
  the case ${\bf C}=\Lambda$ Connes' cyclic category, can easily be
  obtained in the relative derived category of cyclic modules ${\bf
    D}(\lmod{(\Lambda,{\bf T})})$, and with some work~\cite[Lem. 5.2,
  Lem. 5.3]{Kaygun:CupProduct} also in ${\bf D}(\lmod{\C{M}})$ the
  derived category of mixed complexes.  Thus followed by the induced
  map of $\phi_\bullet$ in cohomology we obtain similar pairings
  defined in ${\bf D}(\lmod{(\Lambda,{\bf T})})$ and ${\bf
    D}(\lmod{\C{M}})$.  In \cite[Thm. 5.4]{Kaygun:CupProduct} we
  showed that the pairing we construct here and the pairing
  constructed in the derived category of mixed complexes are naturally
  isomorphic.  Now, Proposition~\ref{Comparison} gives us the natural
  isomorphism between the pairing we construct here and the pairing
  constructed in the relative derived category of cyclic modules.

  Our aim is to show that pairings defined in
  \cite{ConnesMoscovici:HopfCyclicCohomology,
    Gorokhovsky:SecondaryCharacteristicClasses,
    Crainic:CyclicCohomologyOfHopfAlgebras,
    KhalkhaliRangipour:CupProducts, Kaygun:CupProduct,
    Sharygin:CupProducts, Rangipour:CupProductsI} in Hopf-cyclic
  cohomology are naturally isomorphic as natural transformations of
  double functors.  There are certain variations between these
  pairings: The Connes-Moscovici, Gorokhovsky and Crainic pairings are
  defined for $C=H$, $q=0$ and only for $M=k_{\sigma,\delta}$ the
  canonical 1-dimensional SAYD module associated with a modular pair
  involution in $H$.  The Rangipour-Khalkhali pairing, and the pairing
  Rangipour defined in \cite{Rangipour:CupProductsI}, are defined for
  an arbitrary module coalgebra $C$ acting equivariantly on a module
  algebra $A$, and for arbitrary bi-degree $(p,q)$ with an arbitrary
  SAYD module as coefficients~\cite{HKRS:SaYDModules}.  Finally, the
  pairing we defined in \cite{Kaygun:CupProduct} and here in
  Theorem~\ref{CohomologyPairing} work in the same setup as the
  Khalkhali-Rangipour and Rangipour pairings but we allow arbitrary
  coefficient module/comodules.

  The original pairing in Hopf-cyclic cohomology as defined by Connes
  and Moscovici~\cite{ConnesMoscovici:HopfCyclicCohomology} is
  constructed on the (co)cyclic module level, and $(b,B)$-complex is
  utilized to compute the Hopf-cyclic classes used as its input.  This
  is done in the derived the category of mixed
  complexes~\cite{JonesKassel:BivariantCyclicTheory}.  Crainic, and
  later Nikonov and Sharygin defined their version of the pairing
  using Cuntz-Quillen formalism of
  $X$-complexes~\cite{CuntzQuillen:NonsingularityII}.  This is done in
  the homotopy category of towers of super complexes which is homotopy
  equivalent to the derived category of mixed complexes by
  Quillen~\cite{Quillen:CyclicHomologyType}.  In their setup
  Gorokhovsky~\cite{Gorokhovsky:SecondaryCharacteristicClasses}, and
  later Khalkhali and Rangipour~\cite{KhalkhaliRangipour:CupProducts}
  also used mixed complexes to obtain their cohomology classes, and
  ($H$-invariant) closed graded (co)traces to implement their
  pairings.  This is akin to Connes' use of closed graded traces to
  implement the ordinary cup product in cyclic cohomology \cite[III.1,
  Thm. 12]{Connes:Book}.  In \cite{Kaygun:CupProduct} we used both the
  derived category of (co)cyclic modules and the derived category of
  mixed complexes to construct pairings as natural transformations of
  derived double functors, and we defined a comparison natural
  transformation between these derived functors which were
  isomorphisms in the cases we are interested.  Independently,
  Rangipour developed another version of the cup product on the level
  of (co)cyclic modules~\cite{Rangipour:CupProductsI} similar
  to~\cite{Kaygun:CupProduct}.

  In Theorem~\ref{Equivalent}, we gave an interpretation of the cyclic
  cohomology computed via cyclic invariants of Hochschild cocycles as
  a derived functor using a relative derived category.  Then in
  Proposition~\ref{Comparison} we defined a comparison natural
  transformation between ordinary and relative derived functors which
  is an isomorphism.  The primary reason we are interested in relative
  cyclic cohomology is the fact that ($H$-invariant) closed graded
  traces are in one-to-one correspondence (cf. \cite[III.1$\alpha$,
  Prop. 4]{Connes:Book} and \cite[Lem. 2.2. and
  Lem. 2.3]{KhalkhaliRangipour:CupProducts}) with cyclic- and
  $H$-invariant Hochschild cocycles which are used to implement some
  of the pairings we enumerated above.  Thus
  Proposition~\ref{Comparison} provides the crutial comparison natural
  transformation between the pairing we define in here and
  \cite{Kaygun:CupProduct}, and aforementioned pairings.
\end{remark}

\begin{theorem}
  Let $A$ be a $H$-module algebra and $C$ be a $H$-module coalgebra
  acting on $A$ equivariantly over $H$.  The pairings defined in
  \cite{ConnesMoscovici:HopfCyclicCohomology,
    Gorokhovsky:SecondaryCharacteristicClasses,
    Crainic:CyclicCohomologyOfHopfAlgebras,
    KhalkhaliRangipour:CupProducts, Kaygun:CupProduct,
    Sharygin:CupProducts, Rangipour:CupProductsI} are naturally
  isomorphic as natural transformations of isomorphic double functors.
\end{theorem}

\begin{proof}
  All of the pairings enumerate above are composed of two parts
  \[ \Ext^p(k_\bullet,X_\bullet)\otimes\Ext^q(Y_\bullet,k_\bullet) 
    \xra{\eta^{p,q}}
     \Ext^{p+q}(diag_\bullet\Hom_k(X_\bullet,Y_\bullet),k_\bullet)
     \xra{\Ext^{p+q}(\phi,k_\bullet)} \Ext^{p+q}(Z_\bullet,k_\bullet)
  \] 
  \begin{enumerate}
  \item An {\em external part} which mixes a Hopf-cyclic cohomology
    class of the module coalgebra $C$ and a Hopf-cyclic cohomology
    class of the module algebra $A$, as we do in
    Proposition~\ref{DiagonalHom}, to produce an abstract cyclic
    cohomology class which is not necessarily a Hopf-cyclic class of a
    module (co)algebra

  \item An {\em internal part} which interprets the new class we
    obtained as an ordinary cyclic cohomology class of the module
    algebra $A$ as we do in Theorem~\ref{CohomologyPairing}.
  \end{enumerate}
  Lemma~\ref{Agreement} allows us to conclude that the external parts
  of all such pairings agree everywhere provided that they agree on
  the bi-degree $(p,q) = (0,0)$.  So, we consider two Hopf-cyclic
  classes $\alpha\in HC^0_{\rm Hopf}(C,M) =
  \Ext_{\Lambda}^0(k_\bullet, C_\bullet(C,M))$ and $\beta\in HC^0_{\rm
    Hopf}(A,M) = \Ext_\Lambda^0(C_\bullet(A,M),k_\bullet)$.  The first
  class $\alpha$ can be represented with a $k$-linear morphism
  $\alpha'\colon A\otimes M\to k$ in $\Hom_k(A\otimes M, k)$ and the
  second via an element $\beta' = \sum_i c_i\otimes m_i$ in $C\otimes
  M$ which have invariance properties with respect to the diagonal
  action of $H$.  Then the same formula which defines
  $HC^0(\phi_\bullet)$
  \[ (\alpha\smile\beta)(f) = \sum_i \alpha'(f(c_i)\otimes m_i) \] for
  the specific case $f\in \Hom_k(C,A)$ given by $f(c):= c\tr a$ with
  $a\in A$, is used by Connes-Moscovici~\cite[VIII,
  Prop.1]{ConnesMoscovici:HopfCyclicCohomology}, by
  Gorokhovsky~\cite[Sect. 3,
  Eq. 3.11]{Gorokhovsky:SecondaryCharacteristicClasses}, by
  Crainic~\cite[Sect. 4.6,
  Eq. 20]{Crainic:CyclicCohomologyOfHopfAlgebras}, by
  Khalkhali-Rangipour~\cite[Sect. 5]{KhalkhaliRangipour:CupProducts},
  and by Nikonov-Sharygin~\cite[Sect. 3.3]{Sharygin:CupProducts}.
  Thus we also observe that these pairings use the same internal part
  $\phi_\bullet\colon Cyc_\bullet(A)\to
  diag_\bullet\Hom_k(C_\bullet(C,M),C_\bullet(A,M))$ which comes from
  the fact that $C$ acts on $A$ equivariantly by \cite[Prop. 2.4 and
  Prop. 2.7]{Kaygun:CupProduct}.  In~\cite{Rangipour:CupProductsI}
  Rangipour splits his cup product into two pieces as we do here and
  in~\cite{Kaygun:CupProduct}. The external part of his pairing
  defined in \cite[Sect. 2, Eq. 2.11]{Rangipour:CupProductsI}, and the
  internal part defined in~\cite[Sect. 2,
  Eq. 2.13]{Rangipour:CupProductsI} are identical with ours.  Note
  that even though the formulae agree, the computations are performed
  in different derived categories.  Not all of these categories are
  homotopy equivalent but we have comparison natural transformations
  which are isomorphisms between the derived double functors evaluated
  on the objects we are interested~(Proposition~\ref{Comparison} and
  \cite[Thm. 5.4]{Kaygun:CupProduct}).
\end{proof}

\begin{remark}
  The pairing we defined in Theorem~\ref{CohomologyPairing} can be
  easily extended to the periodic Hopf-cyclic cohomology.  Note that
  cyclic cohomology groups computed here either via the derived
  functors of $\Hom_\Lambda(\ \cdot\ ,k_\bullet)$ or
  $\Hom_\Lambda(k_\bullet, \ \cdot\ )$ are naturally graded modules
  over the graded algebra $\Ext_\Lambda^*(k_\bullet,k_\bullet)$ which
  is a polynomial algebra over one generator of degree $\pm
  2$~\cite[Cor. 7]{Connes:ExtFunctors}, which we will denote by $S$.
  This generator implements the periodicity
  operator~\cite[Lem. 8]{Connes:ExtFunctors}, which really is a
  natural transformation of functors of the form $S\colon HC^p (\
  \cdot\ )\to HC^{p\pm 2}(\ \cdot\ )$.  Now using
  \cite[Cor. 1.4]{Nistor:BivariantChernConnes} we conclude that our
  pairing is a morphism of $S$-modules, i.e. compatible with the
  periodicity morphism.  Or we can use
  \cite[Thm. 14]{Sharygin:CupProducts} to prove the pairing defined in
  the derived category of mixed complexes is a morphism of
  $S$-modules, and we transport the action to the pairing defined in
  the derived category of $\Lambda$-modules.  This means the functor
  $HC^*$ in the pairing we defined above can be replaced by $HP^*$ to
  obtain a periodic version.
\end{remark}


\end{document}